\documentclass[pdflatex,sn-mathphys-num]{sn-jnl}

\usepackage{graphicx}%
\usepackage{multirow}%
\usepackage{amsmath,amssymb,amsfonts}%
\usepackage{amsthm,amssymb,mathtools}%
\usepackage{mathrsfs}%
\usepackage[title]{appendix}%
\usepackage{xcolor}%
\usepackage{textcomp}%
\usepackage{manyfoot}%
\usepackage{booktabs}%
\usepackage{algorithm}%
\usepackage{algorithmicx}%
\usepackage{algpseudocode}%
\usepackage{listings}%
\usepackage{enumitem}

\theoremstyle{thmstyleone}%
\newtheorem{theorem}{Theorem}
\newtheorem{proposition}[theorem]{Proposition}%

\theoremstyle{thmstyletwo}%
\newtheorem{remark}{Remark}%

\theoremstyle{thmstylethree}%
\DeclareMathOperator{\ran}{ran}

\DeclareMathOperator{\coker}{coker}

\newcommand{\orders}[1]{[#1]}
\numberwithin{equation}{section}

\newcommand{\commRV}[1]{{\leavevmode\color{black}#1}} 

\renewcommand\footnotemark{}
\makeatletter
\renewcommand{\@biblabel}[1]{[#1]} 
\renewcommand\cite{\citep}          
\makeatother

\begin{document}

\title[Inversion of an analytic operator function through Fredholm quotients and its application]{Inversion of an analytic operator function through Fredholm quotients and its application}


\author*[1]{\fnm{Won-Ki} \sur{Seo}}



\affil*[1]{\orgdiv{School of Economics}, \orgname{University of Sydney}}
\newcommand{\revlab}[1]{\phantomsection\label{#1}} 



\abstract{We characterize the inverse of an analytic Fredholm operator-valued function $A(z)$ near an isolated singularity within a general Banach space framework. Our approach relies on the sequential factorization of $A(z)$ via Fredholm quotient operators. By analyzing the properties of these quotient operators near $z=z_0$, we fully characterize the Laurent series expansion of the inverse of  $A(z)$ in terms of its Taylor coefficients around the singularity. These theoretical results are subsequently applied to characterize the solution of a general autoregressive law of motion in a Banach space.}




\maketitle

	\section{Introduction}\label{intro}
Let \(\mathcal B\) be a complex Banach space and \(\mathcal L_{\mathcal B}\) be the space of bounded linear operators acting on \(\mathcal B\) with operator norm $\|\cdot\|_{\mathcal L_{\mathcal B}}$.  Let \(A(z)\) be an analytic \(\mathcal L_{\mathcal B}\)-valued function defined on an open and connected subset $U$ of \(\mathbb{C}\). We then consider the Taylor series of $A(z)$ around \(z_0 \in U\) as follows:
\begin{equation}\label{eq01}
A(z)  = \sum_{j= 0}^\infty A_{j,z_0}(z-z_0)^j, \quad A_{j,z_0} \in \mathcal L_{\mathcal B}, \quad z\in U.
\end{equation}  
This paper studies the inversion of the above operator-valued function when $z_0$ is an isolated singularity and the derivation of the expression of $A(z)^{-1}$ around $z=z_0$. Instances of this problem arise in various applications, such as asymptotic linear programming [\citealp{lamond1989generalized, lamond1993efficient}], Markov chains [\citealp{avrachenkov1999fundamental}], and linear control theory [\citealp{howlett1982input}]. Additional examples can be found in \citet[Section 1]{GONZALEZRODRIGUEZ2015307} and \citet[Section 1]{FRANCHI20112896}, and we also discuss a specific example in Section \ref{sec_motive}.

Specifically, this paper aims to provide a closed-form expression for \(A(z)^{-1}\) around an isolated singularity $z_0$ when \(A(z)\) is a Fredholm-valued function (i.e.,\(\dim(\ker A(z))<\infty\) and \(\dim(\coker A(z))<\infty\) for \(z\in U\)). We provide a recursive formula for determining the coefficients in the Laurent series of \(A(z)^{-1}\) around $z=z_0$, in terms of the Taylor coefficients \(\{A_{j,z_0}\}_{j\geq 1}\) appearing in \eqref{eq01} and the associated projections, for any order of the pole at \(z=z_0\); moreover, we also demonstrate that the order of the pole can be characterized by those operators. These results are achieved by the factorization of a Fredholm operator-valued function into a perturbed identity operator and a Fredholm quotient, along with a closed-form expression for the latter in terms of \(\{A_{j,z_0}\}_{j\geq 1}\). As detailed in Section \ref{sec_liter}, this approach not only distinguishes our inversion results from existing ones but also provides a useful recasting of previous findings in a Hilbert/Euclidean space setting.

\subsection{Previous Work}\label{sec_liter}

\commRV{
Inversion of an operator-valued function has been studied in various contexts. Particularly, when $A(z)$ is a linear operator pencil (i.e., $A_j=0$ for $j\geq 2$), much is already known in a Hilbert space setting [\citealp{FRANCHI2020238,HOWLETT200968}] and in a more general Banach space setting [\citealp{albrecht2011necessary,ALBRECHT2014411,albrecht2019fundamental,ALBRECHT202033}], and these methods often extend to more general cases. Specifically, although \cite{ALBRECHT2014411} and \cite{HOWLETT200968} primarily focus on linear pencils, they provide a detailed framework for reformulating polynomial pencils as linear ones on an augmented space. This approach yields a closed-form expression for the desired inverse. Subsequently, the results in the aforementioned papers were extended to cases involving essential singularities \citep{Albrecht2020JAMS, ALBRECHT202033}. The theoretical results in \cite{howlett1982input} extensively cover the inversion of polynomial or analytic matrix-valued functions; see also \cite{avrachenkov1999fundamental} for a study on singularly perturbed finite state Markov chains. Inversion results for analytic operator-valued functions in a finite-dimensional setting can also be found in \citep{Avrachenkov2001, FRANCHI20112896, franchi2016, GONZALEZRODRIGUEZ2015307}. A more extensive treatment of this topic, covering both finite-dimensional and infinite-dimensional (Hilbert or Banach) spaces, can be found in \cite{AvrFiHow13} (see Chapters 2 and 8, respectively).

 More recently, \cite{BS2018} and \cite{Franchi2017b} consider the inversion of an analytic Fredholm-valued function in a Hilbert space setting in order to characterize solutions to an autoregressive (AR) law of motion applied to function-valued random elements. Subsequently, \cite{albrecht2021resolution} developed a more general characterization of the solution in a more general setting (including essential singularities) concerning polynomial operator-valued functions. There are also some results on the inversion of $A(z)^{-1}$ in a Banach space setting for Fredholm-valued functions [\citealp{Gohberg2013,seo_2023_fred}] and for possibly non-Fredholm-valued functions [\citealp{seo_2022}]. However, the approach for obtaining a closed-form expression of $A(z)^{-1}$ in the present paper is, as will become apparent, not only distinct from existing ones, but also the theoretical results to be developed do not have substantial overlap with the existing literature.}


\subsection{An Example}\label{sec_motive}
Suppose that a random sequence $X_t$, taking values in a Banach space, satisfies an AR law of motion given by 
\begin{equation} \label{eq001}
\sum_{j=0}^\infty A_{0,j} X_{t-j}  = \varepsilon_t, \quad t\geq 1,
\end{equation}
where $\varepsilon_t$ is another sequence of random elements that typically exhibits simpler dynamics than $X_t$, such as an independently and identically distributed sequence.
Then the operator-valued function 
\begin{equation} \label{eq001add}
A(z) = \sum_{j=0}^\infty A_{0,j} z^j
\end{equation}
 is called the characteristic \commRV{series}\revlab{r1_major3a} of the AR law of motion. The case where $A(z)$ is invertible for every $z$ satisfying $|z|\leq 1 +\eta$ for some $\eta>0$ but not for $z=1$ has received considerable attention in the literature on time series analysis. In this case, it turns out that the behavior of solutions to the AR law of motion \eqref{eq001} crucially depends on the local behavior of $A(z)^{-1}$ around $z=1$ (see e.g., \cite{Schumacher1991}). A complete characterization of such solutions (except for the initial condition at time zero) reduces to answering the following two questions: (i) what is the order of the pole at $z=1$, and (ii) what is an explicit expression for the Laurent series of $A(z)^{-1}$ around $z=1$?

When $\varepsilon_t$ is understood as a stationary sequence, (i) and (ii) are the key questions of the so-called Granger-Johansen representation theory in the literature on time series analysis, which has been studied either in a (finite-dimensional) Euclidean space [\citealp{BeareHowlett2026,engle1987,Johansen1991,johansen1992representation,franchi2016,Franchi2017a}]  or a Hilbert space [\citealp{BSS2017,BS2018,Franchi2017b}] or a Banach space [\citealp{seo_2022,seo_2023_fred,albrecht2021resolution}] setting. A recent work by \cite{seo_2023_fred} appears to be successful to some extent in this direction, as it provided answers to (i) and (ii) in a Banach space setting when \(A(z)\) is Fredholm and the pole order is one or two. As shown in that paper, answering these questions requires solving a certain system of operator equations using properly defined generalized inverses. However, since those results are derived using operator algebra after specifying the pole order as one or two, they do not, in general, characterize \(A(z)^{-1}\) near a pole of any arbitrary order in a unified framework. The present paper overcomes this limitation by providing a unified approach to characterizing \(A(z)^{-1}\).

We have only considered the case where \(A(z)\) has a real unit root (i.e., \(A(z)\) is not invertible at \(z=1\)). Additionally, there are other papers concerning different cases, such as a complex unit root (see, e.g., \cite{gregoir1999multivariate,bierens2001complex}),  with potential applications in statistical time series analysis. In this literature, the local behavior of \(A(z)^{-1}\) is crucial for understanding the behavior of the random sequence determined by \eqref{eq001}.

\section{Fredholm quotients}\label{sec_facto}
We hereafter always assume that \(A(z)\) is an analytic (i.e., complex-differentiable) operator-valued function given in \eqref{eq01} and, for each \(z \in U\), \(A(z)\) is a Fredholm operator acting on a separable Banach space \(\mathcal B\), where \(U\) is an open and connected set. 
The Fredholm property of \(A(z)\) aligns with assumptions made in some preceding articles (see, e.g., \cite{BS2018,Franchi2017b}) that study the inversion of analytic operator-valued functions. It is noteworthy that any linear operator acting on a finite-dimensional vector space is Fredholm.

In practice, it is of particular interest to characterize \(A(z)^{-1}\) using the coefficients \(\{A_{j,z_0}\}_{j \geq 0}\) in \eqref{eq01}, and this is precisely the focus of the present paper. 
The subsequent theoretical results primarily rely on the factorization of the Fredholm-valued function, as described in the following proposition: 
	\begin{proposition}\label{propo1}  \commRV{Suppose that \(A(z)\) in \eqref{eq001add} is invertible for any $z\in U \setminus \{z_0\}$ but not at $z=z_0 \in U$}. Let \(P_1\) denote any projection on \(\ran A(z_0)\) and let \(Q_1 = I - P_1\). Then there exists an analytic index-zero Fredholm operator-valued function \(A^{\orders{1}}(z)\) satisfying
		\begin{equation} 
		A(z) = [P_1 + (z-z_0)Q_1]A^{\orders{1}}(z) \label{fredfac}
		\end{equation}
\commRV{for $z \in U\setminus\{z_0\}$. 
$A^{\orders{1}}(z)^{-1}$ has a pole of order $d-1$ at $z=z_0$ if $A(z)^{-1}$ has a pole of order $d$ at $z=z_0$. } 
	\end{proposition}
\revlab{r1_major4}\commRV{
\begin{remark}\label{remadd1}
Several insightful observations and discussions, which complement Proposition \ref{propo1}, were provided by an anonymous referee. First, while Proposition \ref{propo1} implicitly assumes that the projections $P_1$ and $Q_1$ are well-defined, this does not pose any theoretical issues; any Fredholm operator has a finite-dimensional cokernel, which is a sufficient condition for the existence of such projections (see \cite{megginson1998}, Theorem 3.2.18). Second, the fact that $A(z)$ is invertible for $z \in U \setminus \{z_0\}$ implies that $A(z)$ in Proposition \ref{propo1} is necessarily Fredholm of index zero for all $z \in U$. Third, as discussed in the proof, the operator-valued function $A^{\orders{1}}(z)$ in Proposition \ref{propo1} is given by $A^{\orders{1}}(z) = [P_1 + (z-z_0)^{-1}Q_1] A(z)$ for $z \in U \setminus \{z_0\}$, ensuring that \eqref{fredfac} holds. This necessarily implies that if $A^{\orders{1}}(z)$ is invertible, then $A(z)^{-1} = A^{\orders{1}}(z)^{-1} [P_1 + (z-z_0)Q_1]$; consequently, $A(z)^{-1}$ has a pole of order 1 at $z=z_0$. Combined with \eqref{eq001add}, the above expression for $A^{\orders{1}}(z)$ implies that $A^{\orders{1}}(z)$ can be written as $A^{\orders{1}}(z) = \sum_{j=0}^\infty (z-z_0)^j A^{\orders{1}}_{j,z_0}$ (where $A^{\orders{1}}_{j,z_0} = P_1 A_{j,z_0} + Q_1 A_{j+1,z_0}$), thereby establishing the analyticity of $A^{\orders{1}}(z)$ and providing an alternative to the corresponding proof in the Appendix. Note also that, based on the aforementioned expression, $A^{\orders{1}}(z)$ can be written as $A^{\orders{1}}(z) = A_{0,z_0}^{\orders{1}} + S(z)$, where $S(z)=\sum_{j=1}^{\infty}(z-z_0)^jA^{\orders{1}}_{j,z_0}$. Then, from a Neumann series expansion, the invertibility of $A_{0,z_0}^{\orders{1}}$ implies that $A^{\orders{1}}(z)^{-1}$ exists for $z$ close enough to $z_0$ (specifically, $A^{\orders{1}}(z)^{-1} = (A_{0,z_0}^{\orders{1}})^{-1}[I+S(z)(A_{0,z_0}^{\orders{1}})^{-1}]^{-1}=\sum_{\ell=0}^\infty (A_{0,z_0}^{\orders{1}})^{-1}(-S(z)(A_{0,z_0}^{\orders{1}})^{-1})^\ell$), which implies that $A^{\orders{1}}(z)$ is Fredholm of index zero. 

\end{remark}
}


The above proposition may be understood as a natural extension of some existing results developed in a Hilbert space setting to a Banach space setting; see e.g., Theorem 2.4 of \cite{behrndt2015} and Theorem 3.4 of \cite{Gesztesy2015}.  \cite{HOWLAND197112} obtains a similar result when \(A(z)\) satisfies an additional condition. 
Given that \(I=P_1 + Q_1\), the operator \(P_1 + (z-z_0)Q_1\) is understood as a perturbed identity near $z=z_0$.  Its inverse is simply given by \(P_1 + (z-z_0)^{-1}Q_1\), and thus \(A^{\orders{1}}(z)\) may be regarded as the quotient operator (see, e.g., \cite{kaufman1978representing,KOLIHA2014688}) obtained by dividing \eqref{fredfac} by the perturbed identity. It is also important to note that the resulting operator \commRV{\(A^{\orders{1}}(z)\)} is similar to \(A(z)\) in the sense that both are analytic Fredholm-valued functions, but differentiated in that \commRV{\(A^{\orders{1}}(z)\)} has a pole of order \(d-1\) at \(z=z_0\). This makes it possible to apply the same factorization repeatedly, for $k=1,\ldots, d$ with $d > 0$,
\begin{equation} \label{eqfreadholmq}
A^{\orders{k-1}}(z) =(P_{k}+ (z-z_0)Q_{k}) A^{\orders{k}}(z), \quad \text{with} \quad  A^{\orders{0}}(z)=A(z),
\end{equation}
until we obtain an analytic operator-valued function $A^{\orders{d}}(z)$, which is invertible at \(z=z_0\). Noting that, for each \(z\), \(A^{\orders{k}}(z)\) can be understood as a quotient operator and that the order of the pole of \(A^{\orders{k}}(z)^{-1}\) at \(z = z_0\) is reduced by \(k\) compared to that of \(A^{\orders{0}}(z)^{-1}=A(z)^{-1}\), we will refer to \(A^{\orders{k}}\) as the Fredholm quotient of order \(k\) hereafter. \revlab{r1_major5}\commRV{The repetition described by \eqref{eqfreadholmq} leads to the following generalization of Proposition \ref{propo1}.}

\begin{proposition}[Fredholm factorization]\label{propo2} Let the assumptions of Proposition \ref{propo1} hold and let \revlab{r1_major6}$\commRV{A^{\orders{0}}(z) = A(z)}$. Then for some $d>0$  
				\begin{align}
				A(z) = (P_1+ (z-z_0)Q_1) \cdots(P_{d} + (z-z_0)Q_{d})A^{\orders{d}}(z), \label{fredfac2}
				\end{align}
				where $A^{\orders{d}}(z)$ is analytic, $A^{\orders{d}}(z_0)$ is invertible, $\ran P_j = \ran A^{\orders{j-1}}$ and $P_j = I - Q_j$ for $j=1,\ldots,d$. 
				Moreover,
				\begin{align} \label{eqinclusion}
				\ran P_{1} \subset \ran P_{2} \subset \cdots \subset \ran P_{d}.
				\end{align}		
\end{proposition}
Similar results to the above proposition, considering the case where \(\mathcal B\) is a Hilbert space, can be found in \cite{behrndt2015} and \cite{Gesztesy2015}.  Proposition \ref{propo2}  will be a key input to the subsequent discussion.
Note that if $A^{\orders{d}}(z_0)$ is invertible, then it is surjective. This implies that any additional factorization given by \eqref{eqfreadholmq} when applied to \eqref{fredfac2}, changes nothing and yields $A^{\orders{d}}(z)=A^{\orders{d+1}}(z_0)=\cdots$ since $P_{d+1}=P_{d+2}=\cdots=I$ (and thus $Q_{d+1}=Q_{d+2}=\cdots =0$). \revlab{r1_major7}\commRV{If $P_d \neq I$ and $P_{d+1}=I$ in Proposition \ref{propo2}, then $d$ in \eqref{fredfac2} is the (uniquely determined) order of the pole at $z=z_0$.}


\begin{remark} \label{rem01}
Note that $P_{k}$ is given by the projection onto $\ran A^{\orders{k-1}}(z_0)$ in Proposition \ref{propo2} for $k=1,\ldots,d$, and there are no other requirements; that is, $\ker P_k$ can be any arbitrarily convenient choice. This means that, in a Hilbert space setting, $P_k$ can be set as an orthogonal projection whose kernel is given by $[\ran A^{\orders{k-1}}]^\perp$.
\end{remark}

The Fredholm quotient of order $k$, $A^{\orders{k}}$, and the projection $P_{k+1}$ onto its range play a crucial role in the subsequent discussion. These can further be characterized in terms of the Taylor coefficients $\{A_{j,z_0}\}_{j\geq 0}$ of $A(z)$ via the recursive formula to be given in Proposition \eqref{propo3} below. Hereafter, for the convenience in presenting the subsequent theoretical results, we assume that \revlab{r1_major8}
\begin{equation} \label{eqpcondi}
P_0 = 0, \quad P_{d} \neq I, \quad \text{and} \quad P_{d+j} = I \quad \commRV{\text{for $j > 0$}}, 
\end{equation}  
where, as discussed above, the latter two conditions are not restrictions but should be understood as normalizations that allow us to interpret the integer $d$ appearing in Proposition \ref{propo2} as the pole order of  $A(z)^{-1}$ at $z=z_0$.

\begin{proposition}[Fredholm quotient of order $k$]\label{propo3} Let everything be as in Proposition \ref{propo2} with $d>0$, and suppose \eqref{eqpcondi} holds. Then, for $k=1,\ldots,d$, the Fredholm quotient of order $k$ is given by 
	\begin{equation}
A^{\orders{k}}(z_0) = \sum_{j=0}^{k-1}(P_{j+1}-P_j)A_{j,z_0} + (I-P_k)A_{k,z_0} \label{fredfactor}
	\end{equation}
and for all $0\leq j,\ell \leq k-1$ with $j\neq \ell$
	\begin{align} \label{eqprojec}
(P_{j+1}-P_j)(P_{\ell+1}-P_\ell) = 0 \quad \text{and} \quad (P_{j+1}-P_j)(I-P_k) = 0.
	\end{align}
\end{proposition}

From \eqref{fredfactor}, \eqref{eqprojec} and the facts that $\ran P_k = \ran A^{\orders{k-1}}(z_0)$ and $ \sum_{j=0}^{k-2}(P_{j+1}-P_j) + (I-P_{k-1}) = I$ (with $P_0=0$), we find that, for $k=1,\ldots,d$,
\begin{align}
\ran P_{k} &=  \sum_{j=0}^{k-2}(P_{j+1}-P_j) \ran A^{\orders{k-1}}(z_0) + (I-P_{k-1})\ran A^{\orders{k-1}}(z_0) \notag\\ &= \ran A_{0,z_0} + \ran (P_2-P_1)A_{1,z_0} + \cdots + \ran (I-P_{k-1}) A_{k-1,z_0} \label{eqadg}.
\end{align}
Moreover, noting that $\ran P_{d+1} = \ran A^{\orders{d}}(z_0)  = \mathcal B$, we also find that 
\begin{align}
\ran P_{d+1}  &=  \sum_{j=0}^{d-1}(P_{j+1}-P_j) \ran A^{\orders{d}}(z_0) + (I-P_{d})\ran A^{\orders{d}}(z_0) \notag \\
 &= \ran A_{0,z_0} + \ran(P_2-P_1)A_{1,z_0} + \cdots + \ran (I-P_{d}) A_{d,z_0} = \mathcal B. \label{eqadg2}
\end{align} 

The results given by Propositions \ref{propo1}-\ref{propo3} and \eqref{eqadg}-\eqref{eqadg2} give us a few necessary and sufficient conditions for $A(z)^{-1}$ to have a pole of order $d$ at $z=z_0$.
\begin{proposition}\label{cor1}
Let everything be as in Proposition \ref{propo2} with $d>0$, and suppose \eqref{eqpcondi} holds. The following are equivalent conditions:
\begin{enumerate}[label=(\roman*)]
\item \label{cor1a}$A(z)^{-1}$ has a pole of order $d$ at $z=z_0$
\item\label{cor1b} $A^{\orders{d}}(z_0)$ is invertible while $A^{\orders{d-1}}(z_0)$ is not.
\item \label{cor1c} \commRV{$\mathcal B = \sum_{j=0}^d \ran (P_{j+1}-P_{j}) A_{j,z_0}$  with $P_{d}\neq I$ and  $P_{d+1} = I$.}
\item \label{cor1d} $\{0\} = \cap_{j=0}^d \ker  (P_{j+1}-P_{j}) A_{j,z_0}$ with \commRV{$P_{d}\neq I$ and  $P_{d+1} = I$}. 
\end{enumerate} 
\end{proposition}
We first note that, in the case where $\mathcal B$ is a Hilbert space, the condition (\ref{cor1d}) in Proposition \ref{cor1} can be equivalently written as the following direct sum condition:
\begin{equation} \label{eq001a}
\mathcal B = \sum_{j=0}^d [\ker (P_{j+1}-P_j)A_{j,z_0}]^{\perp}.
\end{equation}
In the literature on time series analysis, when $z_0 = 1$, the conditions given in Proposition \ref{cor1} are often referred to as the I($d$) condition (see, e.g., \cite{johansen2008}), which guarantees that $X_t$, satisfying \eqref{eq001} with the characteristic \commRV{series}\revlab{r1_major3b} equivalent to $A(z)$ in \eqref{eq01} (meaning that the Taylor series of the characteristic \commRV{series}\revlab{r1_major3c} around $z=z_0$ is given by \eqref{eq01}), contains a $d$-th order integrated component $\tilde{\varepsilon}_{d,t} = \sum_{s=1}^{t} \tilde{\varepsilon}_{d-1,s}$ with $\tilde{\varepsilon}_{0,t}=\varepsilon_t$. The conditions (\ref{cor1c}) and (\ref{cor1d}) with $z_0 = 1$ are particularly interesting when compared to the existing characterizations of the I($d$) condition in a potentially infinite-dimensional setting. \cite{Franchi2017b} provide a certain direct sum condition for $X_t$ to contain a $d$-th order integrated component in a Hilbert space setting. \cite{BS2018} provide similar direct sum conditions for $d = 1$ and $2$. These existing conditions, in fact, involve Moore-Penrose inverses of various operators depending on $A_{j,z_0}$ and certain orthogonal projections, and thus are expressed in quite complicated forms unless $d = 1$. However, our direct sum conditions given in (\ref{cor1c}) and \eqref{eq001a} are written in terms of the Fredholm quotients, depending only on $\{A_{j,z_0}\}_{j=0}^d$ and $\{P_j\}_{j=0}^{d+1}$ (furthermore, from Proposition \ref{propo3} and \eqref{eqadg}, $P_j$ can be characterized in terms of $A_{j,z_0}$). Importantly, this formulation applies in a straightforward manner to any arbitrary order $d$, thereby highlighting both its simplicity and generality.  It should also be noted that our conditions in Proposition \ref{cor1} are developed under a more general Banach space setting. To compare the differences between ours and the recent results, see \cite{BS2018} \cite{Franchi2017b}, and \cite{seo_2023_fred}.

\section{Closed-form expression}\label{sec_closed}
In this section, we assume that $A(z)^{-1}$ has a pole of order $d$ at $z=z_0$, and then derive a systematic and complete way to express $A(z)^{-1}$ in terms of $\{A_{j,z_0}\}_{j\geq 0}$ and $\{P_j\}_{j=0}^{d+1}$ (with $P_{d+1}=I$ and $P_0=0$) using the Fredholm quotient $A^{\orders{d}}(z)$. We let 
\begin{equation} \label{eqexpanG}
A^{\orders{d}}(z) = \sum_{j=0}^\infty G_{j,z_0} (z-z_0)^{j}
\end{equation}
and 
\begin{equation}  \label{eqexpanH}
(A^{\orders{d}}(z))^{-1}  = \sum_{j=0}^\infty H_{j,z_0} (z-z_0)^{j}.
\end{equation}
The operators $G_{j,z_0}$ in \eqref{eqexpanG} and $H_{j,z_0}$ in \eqref{eqexpanH} can be explicitly expressed as a function of $\{P_j\}_{j=0}^{d+1}$ and $\{A_{j,z_0}\}_{j\geq 0}$ as follows:
\begin{proposition}\label{propo4}
Let everything be as in Proposition \ref{propo2} with $d>0$, and suppose \eqref{eqpcondi} holds. Then for all $k \geq 0$,
	\begin{align}
	G_{k,z_0}&= \sum_{j=0}^{d}(P_{j+1}-P_{j})A_{k+j,z_0}, \quad j \geq 0, \notag \\ 
H_{j,z_0} &= \begin{cases}
G_{0,z_0}^{-1} \quad &\text{if $j=0$}, \\
- G_{0,z_0}^{-1} \sum_{k=1}^j G_{k,z_0}H_{j-k,z_0}, \quad &\text{if $j\geq1$}.  \label{eqpropo}
\end{cases}
	\end{align}
\end{proposition}\label{propo6}
Combined with the results given in Proposition \ref{propo4}, the following result provides a closed-form representation of $A(z)^{-1}$: 
\begin{proposition}\label{propo5} Let everything be as in Proposition \ref{propo2} with $d>0$, and suppose \eqref{eqpcondi} holds. Then the inverse $A(z)^{-1}$ admits the Laurent series
	\begin{equation}
	A(z)^{-1}= \sum_{j=-d}^{\infty} \Psi_j (z-z_0)^{j},
	\end{equation}
where 
	\begin{equation}
\Psi_j = \begin{cases}
\sum_{\ell=0}^{j+d}H_{j+d-\ell,z_0} (P_{d+1-\ell}-P_{d-\ell}), & \text{if $j <0$},\\
\sum_{\ell=0}^{d} H_{j+\ell,z_0} (P_{\ell+1}-P_{\ell}), & \text{if $j \geq0$}.  
\end{cases} \notag
	\end{equation}
\end{proposition}
\begin{remark} \label{remrevisionadd}
\commRV{
The following insightful observation is entirely due to an anonymous referee, to whom the author is indebted for suggesting an alternative approach to inverting $A(z)$. This approach is based on a closed-form inversion of linear or polynomial operator-valued functions around an isolated singularity, as developed by \cite{ALBRECHT2014411}.
Let $P_d(z) = \prod_{j=1}^d [P_j + (z-z_0)Q_j] A^{\orders{d}}(z_0)$. Following Proposition \ref{propo2}, we can decompose $A(z)$ as $A(z) = P_d(z) + S_d(z)$, where $S_d(z)=\prod_{j=1}^d [P_j + (z-z_0)Q_j][A^{\orders{d}}(z)-A^{\orders{d}}(z_0)]$ which is analytic in a neighborhood of $z = z_0$. Since $A(z)$ and $P_d(z)$ are invertible in a neighborhood of $z=z_0$ (excluding $z=z_0$), the Neumann expansion yields
\begin{equation}
A(z)^{-1} = P_d(z)^{-1}(I+S_d(z)P_d(z)^{-1})^{-1} =  \sum_{\ell=0}^\infty P_d(z)^{-1} (-S_d(z)P_d(z)^{-1})^{\ell}
\end{equation}
for all $z\neq z_0$ with $|z-z_0|$ sufficiently small. As pointed out by the referee, $P_d(z)$ is a polynomial operator-valued function, and a closed-form expression of its inverse can be obtained by the inversion results of \cite{ALBRECHT2014411}, which can be used more generally for polynomial operator-valued functions without requiring the Fredholm property; see Section 7 of \cite{ALBRECHT2014411}. 
}
\end{remark}

%
\section{Application} \label{grtfred}
In this section, we characterize solutions to the AR law of motion \eqref{eq001} with the characteristic \commRV{series}\revlab{r1_major3d} equivalent to $A(z)$ in \eqref{eq01}, using our theoretical results. The cases  $d=1$ or $d=2$ are regarded as empirically relevant in the time series literature, and we focus on these as applications of our theoretical framework. However, since our characterization of the solutions follows directly from the results in Sections \ref{sec_facto} and \ref{sec_closed}, the representation of $X_t$ for $d\geq 3$ requires only a minor modification. 
We hereafter let $\pi_j(k)$ be defined as follows: $\pi_0(k)=1,$ $\pi_1(k)=k$, and $\pi_j(k)=k(k-1)\cdots (k-j+1)/j!$ for $j \geq 2$. The following results show the desired characterization.
\begin{proposition}\label{propapplication}
Assume that the AR law of motion \eqref{eq001} holds with the characteristic \commRV{series}\revlab{r1_major3e} equivalent to $A(z)$ in \eqref{eq01},  which is further assumed to be invertible for any $z$ satisfying $|z|\leq 1 + \eta$ for some $\eta>0$, except at $z=1$. Then the following hold: 
\begin{enumerate}[label=(\roman*)]
\item \label{propapplicationa} If  $A^{\orders{1}}(1) = A_{0,1} - (I-P_1)A_{1,1}$ is invertible, then $d=1$ and $X_t$ can be represented as follows:
	\begin{equation}\label{eqi1repre}
	X_t = \tau_0 + \Psi_{-1} \sum_{s=1}^t \varepsilon_s + \nu_t, \quad t\geq 1,
	\end{equation}
where $\nu_t = \sum_{j=0}^\infty \Phi_j \varepsilon_{t-j}$, $\Phi_j = \sum_{k=j}^\infty (-1)^{k-j} \pi_j(k)\Psi_k$, and $\Psi_j$ is determined as in  Proposition \ref{propo5} with $d=1$.  
\item \label{propapplicationb} If $A^{\orders{2}}(1)  =  A_{0,1}+(P_2-P_1)A_{1,1}+(I-P_2)A_{2,1}$ is invertible while $A^{\orders{1}}(1)$ is not, then $d=2$ and $X_t$ can be represented as follows: for some $\tau_0, \tau_1 \in \mathcal B$,
	\begin{equation*}
	X_t =  \tau_0 + \tau_1 t + \Psi_{-2}\sum_{r=1}^s \sum_{s=1}^t \varepsilon_s + \Psi_{-1}\sum_{s=1}^t \varepsilon_s + \nu_t,
	\end{equation*}
	where $\nu_t = \sum_{j=0}^\infty \Phi_j \varepsilon_{t-j}$, $\Phi_j = \sum_{k=j}^\infty (-1)^{k-j} \pi_j(k)\Psi_k$, and $\Psi_j$ is determined as in  Proposition \ref{propo5} with $d=2$. 
\end{enumerate}
\end{proposition}
Even though it is not the main focus of this paper, it can be shown that $\|\Phi_j\|_{\mathcal L_{\mathcal B}}$ decreases exponentially as $j$ increases, under the assumptions of Proposition \ref{propapplication}; see, e.g., the proofs of Propositions 3.1 and 4.1 in \cite{seo_2023_fred}. The characterization of solutions to the AR law of motion \eqref{eq001}, as in Proposition \ref{propapplication}, is a central problem in the time series analysis literature. As discussed in Section \ref{sec_motive}, the most well-known results in this context are the Granger-Johansen representation theorems. In this regard, Proposition \ref{propapplication} may be viewed as a version of the Granger-Johansen representation theorem, derived from our previous Fredholm factorization and the properties of the associated Fredholm quotients, in a more general setup.

\revlab{r1_major9}\commRV{
\begin{remark}\label{remadd3}
The case $d=1$ in Proposition \ref{grtfred} has received significant attention in the time series literature (see, e.g., \cite{NSS, seo2020functional, seo2019cointegrated,seoshang25}), particularly concerning unit-root-type nonstationarity; this form of nonstationarity may also be considered a reasonable approximation for the persistence of curve-valued time series (see, e.g., \cite{seoshang22}).  In this setting, Propositions \ref{propo5} and \ref{propapplication} imply that the random walk component takes values in the space $\ran \Psi_{-1} = \ran G_{0,1}^{-1}(I-P_1)$. Given that $G_{0,1} = P_1 A_{0,1} + (I-P_1)A_{1,1}$, it can be shown that $\Psi_{-1} = G_{0,1}^{-1}(I-P_1)$ is in fact an extension of the inverse of the operator $P_1 A_{0,1}(I-P_1) : \ran (I-P_1) \to \ran P_1$. Since $P_1$ is the projection onto $\ran A_{0,1}$, this finding is consistent with existing representation results in Hilbert spaces \cite{BS2018}. As previously discussed, both this result and Proposition \ref{grtfred} can be viewed as special cases of a more general representation framework for any $d \in \mathbb{N}\cup\{\infty\}$ in a Banach space setting (as a straightforward extension of this, it can also be shown that our closed form expression when $d=1$, at any isolated singularity, is a special case of the inversion theorem provided by \cite{albrecht2021resolution}) 
 However, for the more complex cases where $d > 1$, a direct comparison between the results in Proposition \ref{propapplication} and the general representation theorems in \cite{albrecht2021resolution} is nontrivial. This difficulty arises primarily from the differing methodologies used to characterize the coefficient operators, specifically $\Psi_{-2}$ and $\Psi_{-1}$. Because the primary objective of this study is to illustrate an alternative approach for obtaining a closed-form inverse of an analytic Fredholm operator-valued function, a detailed investigation on these representation theorems is beyond the scope of the present work and is left for future research. 
\end{remark}}

\commRV{
\section{Conclusion}
This paper investigates the inversion of analytic Fredholm operator-valued functions by characterizing their associated Fredholm quotients, providing an alternative approach for obtaining a closed-form inverse. The primary contribution of this work lies in the simplicity and clarity of the proposed decomposition methodology (via Fredholm quotients), which provides an alternative closed-form expression for the inverse. While there have been fundamental contributions regarding the inversion of linear or polynomial operator-valued functions under broader assumptions within a Banach space setting (see, e.g., \cite{ALBRECHT2014411, albrecht2019fundamental, ALBRECHT202033, albrecht2011necessary, albrecht2021resolution}), to the best of the author's knowledge, the approach presented here has not been previously explored. Our methodology offers a distinct advantage through its transparent decomposition of an operator-valued function via Fredholm quotients.
}

%


\begin{appendices}
\section{Proofs}
\begin{proof}[Proof of Proposition \ref{propo1}]
Since $P_1$ is an orthogonal projection and $Q_1=I-P_1$, the perturbed identity $P_1 + (z-z_0) Q_1$ is invertible for $z \in U \setminus \{z_0\}$, where $U$ is an open and connected set as introduced in Section \ref{intro}, and its inverse is given by 
\begin{align}\label{eqpf00}
[P_1 + (z-z_0) Q_1]^{-1} = P_1 + (z-z_0)^{-1} Q_1.
\end{align}
Let $A^{\orders{1}}(z)$ be an operator-valued function that is defined as $A^{\orders{1}}(z) = [P_1 + (z-z_0) Q_1]^{-1}A(z) =  P_1A(z) + (z-z_0)^{-1}Q_1A(z)$ for $z\in U\setminus\{z_0\}$ and $A^{\orders{1}}(z_0) = P_1A(z_0) + Q_1A_{1,z_0}(z_0)$. Then $A^{\orders{1}}(z)$ is \commRV{analytic} at $z \in U \setminus \{z_0\}$ and also at $z=z_0$ since
 \begin{align}
\frac{A^{\orders{1}}(z)-A^{\orders{1}}(z_0)}{z-z_0}&= P_1\left(\frac{A(z)-A(z_0)}{z-z_0}\right) + Q_1\left(\frac{A(z)-A(z_0)}{(z-z_0)^2} - \frac{A_{1,z_0}(z_0)}{z-z_0} \right),  \label{eqpf01}
\end{align}
where we used the fact that $Q_1A(z)=Q_1(A(z)-A(z_0))$ since $Q_1A(z_0)=0$. Note that the right hand side  of \eqref{eqpf01} converges to $P_1 A_{1,z_0} + Q_1 A_{2,z_0}$ as $z\to z_0$. We next show that $A^{\orders{1}}(z)$ is an index-zero Fredholm operator for all $z\in U$. Note that  
\begin{equation*}
A^{\orders{1}}(z) =  A(z) + K(z), 
\end{equation*}
where 
\begin{equation*}
K(z) = \begin{cases} - Q_1 A(z) + (z-z_0)^{-1} Q_1 A(z)& \quad \text{if $z\neq z_0$},\\
  Q_1A_{1,z_0}(z_0) & \quad \text{if $z= z_0$}.
\end{cases}
\end{equation*}
Since $Q_1$ is a finite dimensional projection due to the Fredholm property of $A(z)$, $K(z)$ is obviously compact for all $z \in U$. This implies that $A^{\orders{1}}(z)$ and $A(z)$ are Fredholm operators of the same index (see e.g., Theorem 3.11 of \cite{Conway1994}).  

We next show that  $A^{\orders{1}}(z)^{-1}$ has a pole of order $d-1$ at $z=z_0$. Note that 
\begin{equation*}
A(z)^{-1} = A^{\orders{1}}(z)^{-1}[P_1 + (z-z_0)^{-1} Q_1],
\end{equation*}
which implies that $A^{\orders{1}}(z)^{-1}$ has a pole of order at most $d-1$ at $z=z_0$.
Moreover, from \eqref{eqpf00} and the fact that $A^{\orders{1}}(z) = [P_1 + (z-z_0) Q_1]^{-1}A(z)$ for any $z \in U\, \setminus\, \{z_0\}$, we have 
\begin{align}
A^{\orders{1}}(z)^{-1} = A(z)^{-1}P_1  + (z-z_0) A(z)^{-1} Q_1. \label{lemeq001}
\end{align}
Note that the second term in \eqref{lemeq001} has a pole of order $d-1$ at $z=z_0$. Moreover, observe that, for any $x \in \mathcal B$, there exists $y$ such that $P_1x=A(z_0)y$. We then note that 
\begin{align}
|z-&z_0|^{d-1}\|A^{-1}(z)P_1x\|  = |z-z_0|^{d-1}\|A^{-1}(z)A(z_0)y\| \notag \\ & \leq  |z-z_0|^{d-1}\|y\| + \left\|(z-z_0)^{d}A^{-1}(z)\left(\frac{A(z_0)y - A(z)y}{z-z_0}\right)\right\|.    \label{lemeq002}
\end{align}
As $z$ approaches $z_0$, the right-hand side of \eqref{lemeq002} is convergent. That is, \eqref{lemeq002} is bounded for $z$ close enough $z_0$. From the uniform boundedness property of a convergent sequence see, e.g.\ \citep[pp.\ 150-151]{Kato1995}, it is deduced from \eqref{lemeq002} that $|z-z_0|^{d-1}\|A^{-1}(z)P_1\|_{\mathcal L_{\mathcal B}} $ is uniformly bounded. This implies that  $A^{-1}(z)P_1$ has a pole of order at most $d-1$. Combining this result with the fact that the second term in \eqref{lemeq001} has a pole of order $d-1$ at $z=z_0$, we conclude that $A^{\orders{1}}(z)^{-1}$ has a pole of order $d-1$ at $z=z_0$. 
\end{proof}
\begin{proof}[Proof of Proposition \ref{propo2}]
	By repeatedly applying Proposition \ref{propo1}, we may obtain the following:  for $j=1,\ldots,d$,
	\begin{equation*}  
	A^{\orders{j-1}}(z) = [P_{j} + (z-z_0)Q_{j}]A^{\orders{j}}(z),
	\end{equation*}
	where $A^{\orders{0}}(z) \coloneqq A(z)$, the existence of \(d<\infty\) is guaranteed by the analytic Fredholm theorem (e.g., Corollary 8.4 in \cite{Gohberg2013}), and $A^{\orders{d}}(z)$ is holomorphic on $U$ and does not have a singularity at $z=z_0$. From Proposition \ref{propo1}, we can choose $P_j$ as the projection on $\ran A^{{\orders{j-1}}}(z_0)$ for $j=1,\ldots,d$. At $z=z_0$, we have 
\begin{equation*}
A^{{\orders{j-1}}}(z_0) = P_{j} A^{\orders{j}}(z_0), \quad j = 1,\ldots,d. 
\end{equation*}
Since $P_j$ is a projection, we have
\begin{align*}
\ran A^{{\orders{j}}}(z_0) &= P_{j} \ran A^{\orders{j}}(z_0) \oplus Q_{j} \ran A^{\orders{j}}(z_0) \\ &= \ran A^{\orders{j-1}}(z_0) \oplus Q_{j} \ran A^{\orders{j}}(z_0),
\end{align*} 
Therefore, $\ran A^{\orders{j-1}}(z_0) \subset \ran A^{\orders{j}}(z_0)$. Since $\ran P_{j+1} = \ran A^{\orders{j}}$, we conclude that \eqref{eqinclusion} holds. 
\end{proof}
\begin{proof}[Proof of Proposition \ref{propo3}]
\eqref{eqprojec} directly follows from \eqref{eqinclusion}, and thus the detailed proof is omitted. Moreover, observing that $P_{j}P_{j+\ell}=P_j$, $P_{j}Q_{j+\ell}=0$,  $Q_{j}P_{j+\ell}=P_{j+\ell}-P_{j}$ and $Q_{j}Q_{j+\ell} = Q_{j+\ell}$ for $\ell\geq 1$, we find that $\Pi_{j=1}^k(P_{j}+(z-z_0)Q_j) =  \sum_{j=0}^{k}(P_{j+1}-P_{j})(z-z_0)^j$, where $P_{k+1} = I$ and $P_0=0$.  
 Since $A^{\orders{k}}(z)$ is holomorphic, we know from Proposition \ref{propo2} that 
	\begin{align} 
	A(z) = &\left[ \sum_{j=0}^{k}(P_{j+1}-P_{j})(z-z_0)^j \right]\left[ \sum_{j=0}^\infty	A^{\orders{k}}_{j,z_0}(z-z_0)^j \right], \label{lameqpf01}
	\end{align}	
where $A^{\orders{k}}_{j,z_0}$ is the coefficient associated with $(z-z_0)^j$ of the Taylor series of $A^{\orders{k}}(z)$ at $z=z_0$. 
	Since $A(z)$ is holomorphic at $z=z_0$, we have
	\begin{equation*}
	A(z) = \sum_{j=0}^{\infty}A_{j,z_0} (z-z_0)^j.
	\end{equation*}
	For $0\leq m \leq k$, we equate the coefficients of $(z-z_0)^m$ in \eqref{lameqpf01} and obtain 
	\begin{equation*}
	A_{m,z_0} =  \sum_{j=0}^m (P_{m+1-j} - P_{m-j}) A^{\orders{k}}_{j,z_0}. 
	\end{equation*}
Since $(P_{j+1}-P_j) (P_{\ell+1}-P_\ell) = 0$ if $j \neq \ell$, we find that the following hold: 
\begin{align}
&(P_1-P_0)A_{0,z_0} = (P_1-P_0)A^{\orders{k}}_{0,z_0}, \notag   \\
&(P_2-P_1)A_{1,z_0} = (P_2-P_1)A^{\orders{k}}_{0,z_0}, \notag \\
&\vdots \notag \\
&(P_{k+1}-P_k)A_{k,z_0} = (P_{k+1}-P_k)A^{\orders{k}}_{0,z_0}. \notag
\end{align}
Since $\sum_{j=0}^{k}(P_{j+1}-P_j) = P_{k+1}-P_0 = I$, we find that 
\begin{align*}
A^{\orders{k}}(z_0) = A^{\orders{k}}_{0,z_0}  = (P_1-P_0)A_{0,z_0}+ \cdots +  (P_k-P_{k-1})A_{k-1,z_0} + (I-P_k)A_{k,z_0},
\end{align*}
as desired.
\end{proof}
\begin{proof}[Proof of Proposition \ref{cor1}]
\ref{cor1a} $\Leftrightarrow$ \ref{cor1b} is an immediate consequence of Proposition \ref{propo2}. We will show \ref{cor1b} $\Rightarrow$ both \ref{cor1c} and \ref{cor1d}, \ref{cor1c} $\Leftrightarrow$ \ref{cor1d},  and  both \ref{cor1c} and \ref{cor1d} $\Rightarrow$ \ref{cor1b}.

First note that 
the invertibility of $A^{\orders{d}}(z_0)=\sum_{j=0}^d (P_{j+1}-P_{j}) A_{j,z_0}$ implies that
  	\begin{equation} \label{eqrequired1}
	\mathcal B =\ran \left(\sum_{j=0}^d (P_{j+1}-P_{j}) A_{j,z_0}\right)  
	\end{equation}
and 
	\begin{equation} \label{eqrequired2}
\{0\} = \ker \left(\sum_{j=0}^d (P_{j+1}-P_{j}) A_{j,z_0}\right).
	\end{equation}
Since $(P_{j+1}-P_j)(P_{\ell+1}-P_{\ell})=0$ unless $j=\ell$ and $\sum_{j=0}^d (P_{j+1}-P_j)=I$, $\ran (P_{j+1}-P_j) =   \ran (P_{j+1}-P_{j}) A_{j,z_0}$ and we find that \eqref{eqrequired1} is equivalent to the direct sum condition given in (iii).  
Moreover, from the fact that $(P_{j+1}-P_j)(P_{\ell+1}-P_{\ell})=0$ for $j\neq \ell$, we find that $x \in \ker (\sum_{j=0}^d (P_{j+1}-P_{j}) A_{j,z_0})$ must satisfy $x \in \ker (P_{j+1}-P_{j}) A_{j,z_0}$ for all $j=1,\ldots,d$. Combining this result with \eqref{eqrequired2}, we find that   
\begin{equation*}
\ker \left(\sum_{j=0}^d (P_{j+1}-P_{j}) A_{j,z_0}\right)= \cap_{j=0}^d \ker  (P_{j+1}-P_{j}) A_{j,z_0}.
\end{equation*}
Thus (ii) $\Rightarrow$ (iii) and (iv). 

Note that $A^{\orders{d}}(z_0) = \sum_{j=0}^d (P_{j+1}-P_{j}) A_{j,z_0}$ is a Fredholm operator of index zero (see Proposition \ref{propo2}), and hence $\dim(\coker A^{\orders{d}}(z_0)) =\dim(\ker A^{\orders{d}}(z_0))$. Thus if either of \eqref{eqrequired1} or \eqref{eqrequired2} is true, then the other is also true. Combining this with the fact that  \eqref{eqrequired1} (resp.\  \eqref{eqrequired2})  is equivalent to \ref{cor1b} (resp.\ \ref{cor1c}) by the properties of $P_j$, we find that \ref{cor1c} $\Leftrightarrow$ \ref{cor1d}. 

If $A^{\orders{d}}(z_0)$ is not invertible, then from the fact that $A^{\orders{d}}(z_0) = \sum_{j=0}^d (P_{j+1} - P_{j}) A_{j,z_0}$, it is straightforward to see that either \eqref{eqrequired1} or \eqref{eqrequired2} cannot hold. Thus, \ref{cor1c} and \ref{cor1d} $\Rightarrow$ \ref{cor1b}.
\end{proof}
\begin{proof}[Proof of Proposition \ref{propo4}]
From \eqref{eqexpanG} and \eqref{lameqpf01}, we observe that 
	\begin{equation*} 
	A(z) = \left[ \sum_{j=0}^d (P_{j+1}-P_{j})(z-z_0)^j \right]\left[ \sum_{j=0}^\infty G_{j,z_0}(z-z_0)^j \right]
	\end{equation*}	
By equating the coefficients of $(z-z_0)^{k}$, we find the following:
\begin{align*}
A_{k,z_0} = \begin{cases}
 \sum_{j=0}^k (P_{k+1-j} - P_{k-j}) G_{j,z_0}  \quad &\text{ if $0\leq k \leq d$},\\
 \sum_{j=0}^d (P_{d+1-j} - P_{d-j}) G_{k-d+j,z_0}  \quad &\text{ if $k \geq d+1$}.\\
\end{cases}
\end{align*}
Since $(P_{j+1}-P_j) (P_{\ell+1}-P_\ell) = 0$ if $j \neq \ell$, we find that the following hold for any $k \geq 0$: 
\begin{align}
&(P_1-P_0)A_{k,z_0} = (P_1-P_0)G_{k,z_0}, \notag  \\
&(P_2-P_1)A_{k+1,z_0} = (P_2-P_1)G_{k,z_0}, \notag\\
&\vdots \notag\\
&(P_{d+1}-P_d)A_{k+d,z_0} = (P_{d+1}-P_d)G_{k,z_0}.  \notag
\end{align}
 Since $\sum_{j=0}^d(P_{j+1}-P_j) = I$, we find that $G_{k,z_0} = \sum_{j=0}^{d}(P_{j+1}-P_j) A_{k+j,z_0}$.

The characterization of $H_{j,z_0}$ in \eqref{eqpropo} can easily be deduced from the identity $G(z)G(z)^{-1} = I$, and hence the detailed proof is omitted. 
\end{proof}
\begin{proof}[Proof of Proposition \ref{propo5}]	
	We know from Proposition \ref{propo2} that 
\begin{equation*}
	A(z)^{-1} =  G(z)^{-1}(P_d + (z-z_0)^{-1} Q_d)\cdots(P_1 + (z-z_0)^{-1} Q_1). 
\end{equation*}
We find that $P_{j+\ell}P_{j}=P_j$, $P_{j+\ell}Q_{j}=P_{j+\ell}-P_{j}$,  $Q_{j+\ell}P_{j}=0$ and $Q_{j+\ell}Q_{j} = Q_{j+\ell}$ for $\ell\geq 1$. {Thus, we find that 
	\begin{align}
&	A(z)^{-1} = G(z)^{-1} \sum_{j=0}^{d}(P_{j+1} - P_j)(z-z_0)^{-j} \notag \\
&= \left(\sum_{j=0}^{\infty} H_j (z-z_0)^{j}\right) \left( \sum_{j=0}^{d}(P_{j+1} - P_j)(z-z_0)^{-j} \right)\notag\\
&= \sum_{j=0}^{d-1} \left(\sum_{\ell=0}^{j}H_{j-\ell} (P_{d+1-\ell}-P_{d-\ell})\right) (z-z_0)^{j-d}   +\sum_{j=0}^{\infty} \left(\sum_{\ell=0}^{d} H_{j+\ell} (P_{\ell+1}-P_{\ell})\right) (z-z_0)^{j}\notag\\
&= \sum_{j=-d}^{-1} \left(\sum_{\ell=0}^{j+d}H_{j+d-\ell} (P_{d+1-\ell}-P_{d-\ell})\right) (z-z_0)^{j}   +\sum_{j=0}^{\infty} \left(\sum_{\ell=0}^{d} H_{j+\ell} (P_{\ell+1}-P_{\ell})\right) (z-z_0)^{j}.
 \notag
	\end{align}
This proves the desired expression of $A(z)^{-1}$.}
\end{proof}
\begin{proof}[Proof of Proposition \ref{propapplication}]	
To show  \ref{propapplicationa}, we apply the equivalent linear filter induced by $(1-z)A(z)^{-1}$ to both sides of \eqref{eq001} with $A(z)$ in \eqref{eq01} (see e.g., the proofs of Theorems 3.1 and 4.1 of \cite{BS2018}). We then know from Propositions \ref{propo4} and \ref{propo5} that  $\Delta X_t:= X_t-X_{t-1} = -\Psi_{-1}\varepsilon_t +  (\nu_t-\nu_{t-1})$, where $\nu_t = \sum_{j=0}^\infty \Phi_j \varepsilon_{t-j}$ and $\Phi_j$ is the Taylor coefficient around $z=0$ associated with $(z-z_0)^j$ of the analytic part of the Laurent series of $A(z)^{-1}$  (see the proof of Proposition 4.1 of \cite{seo_2023_fred}). Then the desired representation of $X_t$ and the expression of $\Psi_j$ given in \ref{propapplicationa} are deduced without difficulty. 

To show  \ref{propapplicationb}, we apply the linear filter induced by $(1-z)^2A(z)^{-1}$, and obtain the desired results in a similar manner. 
\end{proof}

\end{appendices}



\begin{thebibliography}{}
\bibitem[Albrecht et~al., 2014]{ALBRECHT2014411}
Albrecht, A., Howlett, P., and Pearce, C. (2014).
\newblock The fundamental equations for inversion of operator pencils on
  {B}anach space. 
\newblock {\em Journal of Mathematical Analysis and Applications},
  413(1):411--421.  

\bibitem[Albrecht et~al., 2019]{albrecht2019fundamental}
Albrecht, A.~R., Howlett, P.~G., and Verma, G. (2019).
\newblock The fundamental equations for the generalized resolvent of an
  elementary pencil in a unital {B}anach algebra.
\newblock {\em Linear Algebra and its Applications}, 574:216--251.


\bibitem[Albrecht et~al., 2020]{ALBRECHT202033}
Albrecht, A., Howlett, P., and Verma, G. (2020).
\newblock Inversion of operator pencils on {B}anach space using {J}ordan chains
  when the generalized resolvent has an isolated essential singularity.
\newblock {\em Linear Algebra and its Applications}, 595:33--62.


\bibitem[Albrecht et~al., 2020]{Albrecht2020JAMS}
Albrecht, A., Howlett, P., and Verma, G. (2020).
\newblock Inversion of operator pencils on {H}ilbert space.
\newblock {\em Journal of the Australian Mathematical Society}, 108(2):145--176.

\bibitem[Albrecht et~al., 2011]{albrecht2011necessary}
Albrecht, A.~R., Howlett, P.~G., and Pearce, C.~E. (2011).
\newblock Necessary and sufficient conditions for the inversion of
  linearly-perturbed bounded linear operators on {B}anach space using {L}aurent
  series.
\newblock {\em Journal of Mathematical Analysis and Applications},
  383(1):95--110.


\bibitem[Avrachenkov et~al., 2013]{AvrFiHow13}
Avrachenkov, K.~E., Filar, J.~A., and Howlett, P.~G. (2013).
\newblock {\em Analytic Perturbation Theory and Its Applications}.
\newblock SIAM, Philadelphia.

\bibitem[Avrachenkov et~al., 2001]{Avrachenkov2001}
Avrachenkov, K.~E., Haviv, M., and Howlett, P.~G. (2001).
\newblock Inversion of analytic matrix functions that are singular at the
  origin.
\newblock {\em SIAM Journal on Matrix Analysis and Applications},
  22(4):1175--1189.

\bibitem[Avrachenkov and Lasserre, 1999]{avrachenkov1999fundamental}
Avrachenkov, K.~E. and Lasserre, J.~B. (1999).
\newblock The fundamental matrix of singularly perturbed {M}arkov chains.
\newblock {\em Advances in Applied Probability}, 31(3):679--697.

\bibitem{BeareHowlett2026}
{Beare, B. K., and Howlett, P.  (2026).} 
\newblock {The natural components of a regular linear system.} 
\newblock {{\em Oxford Bulletin of Economics and Statistics}, 1--20.} 

\bibitem[Beare et~al., 2017]{BSS2017}
Beare, B.~K., Seo, J., and Seo, W.-K. (2017).
\newblock Cointegrated linear processes in {H}ilbert space.
\newblock {\em Journal of Time Series Analysis}, 38(6):1010--1027.

\bibitem[Beare and Seo, 2020]{BS2018}
Beare, B.~K. and Seo, W.-K. (2020).
\newblock Representation of {I}(1) and {I}(2) autoregressive {H}ilbertian
  processes.
\newblock {\em Econometric Theory}, 36(5):773--802.

\bibitem[Behrndt et~al., 2015]{behrndt2015}
Behrndt, J., Gesztesy, F., Holden, H., and Nichols, R. (2015).
\newblock {On the index of meromorphic operator-valued functions and some
  applications, arXiv:1512.06962v3 [math.SP]}.

\bibitem[Bierens, 2001]{bierens2001complex}
Bierens, H.~J. (2001).
\newblock Complex unit roots and business cycles: Are they real?
\newblock {\em Econometric Theory}, 17(5):962--983.

\bibitem[Conway, 1994]{Conway1994}
Conway, J.~B. (1994).
\newblock {\em A Course in Functional Analysis}.
\newblock Springer.

\bibitem[Engle and Granger, 1987]{engle1987}
Engle, R.~F. and Granger, C. W.~J. (1987).
\newblock Co-integration and error correction: Representation, estimation, and
  testing.
\newblock {\em Econometrica}, 55(2):251--276.

\bibitem[Franchi, 2020]{FRANCHI2020238}
Franchi, M. (2020).
\newblock Some results on eigenvalues of finite type, resolvents and {R}iesz
  projections.
\newblock {\em Linear Algebra and its Applications}, 588:238--271.

\bibitem[Franchi and Paruolo, 2011]{FRANCHI20112896}
Franchi, M. and Paruolo, P. (2011).
\newblock Inversion of regular analytic matrix functions: {L}ocal {S}mith form
  and subspace duality.
\newblock {\em Linear Algebra and its Applications}, 435(11):2896--2912.

\bibitem[Franchi and Paruolo, 2016]{franchi2016}
Franchi, M. and Paruolo, P. (2016).
\newblock Inverting a matrix function around a singularity via local rank
  factorization.
\newblock {\em SIAM Journal on Matrix Analysis and Applications},
  37(2):774--797.

\bibitem[Franchi and Paruolo, 2019]{Franchi2017a}
Franchi, M. and Paruolo, P. (2019).
\newblock A general inversion theorem for cointegration.
\newblock {\em Econometric Reviews}, 38(10):1176--1201.

\bibitem[Franchi and Paruolo, 2020]{Franchi2017b}
Franchi, M. and Paruolo, P. (2020).
\newblock \commRV{Cointegration in functional autoregressive processes.}
\newblock {\em Econometric Theory}, 36(5):803--839.

\bibitem[Gesztesy et~al., 2015]{Gesztesy2015}
Gesztesy, F., Holden, H., and Nichols, R. (2015).
\newblock On factorizations of analytic operator-valued functions and
  eigenvalue multiplicity questions.
\newblock {\em Integral Equations and Operator Theory}, 82(1):61--94.

\bibitem[Gohberg et~al., 2013]{Gohberg2013}
Gohberg, I., Goldberg, S., and Kaashoek, M. (2013).
\newblock {\em Classes of Linear Operators Vol. I}.
\newblock Birkh\"{a}user.

\bibitem[Gonzalez-Rodriguez et~al., 2015]{GONZALEZRODRIGUEZ2015307}
Gonzalez-Rodriguez, P., Moscoso, M., and Kindelan, M. (2015).
\newblock Laurent expansion of the inverse of perturbed, singular matrices.
\newblock {\em Journal of Computational Physics}, 299:307--319.

\bibitem[Gregoir, 1999]{gregoir1999multivariate}
Gregoir, S. (1999).
\newblock Multivariate time series with various hidden unit root, {P}art {I}:
  Integral operator algebra and representation theory.
\newblock {\em Econometric Theory}, 15(4):435--468.

\bibitem[Howland, 1971]{HOWLAND197112}
Howland, J.~S. (1971).
\newblock Simple poles of operator-valued functions.
\newblock {\em Journal of Mathematical Analysis and Applications}, 36(1):12 --
  21.

\bibitem[Howlett, 1982]{howlett1982input}
Howlett, P.~G. (1982).
\newblock Input retrieval in finite dimensional linear systems.
\newblock {\em The ANZIAM Journal}, 23(4):357--382.

\bibitem[Howlett et~al., 2009]{HOWLETT200968}
Howlett, P., Avrachenkov, K., Pearce, C., and Ejov, V. (2009).
\newblock Inversion of analytically perturbed linear operators that are singular at the origin.
\newblock {\em Journal of Mathematical Analysis and Applications}, 353(1):68--84.

\bibitem[Howlett et~al., 2025]{albrecht2021resolution}
Howlett, P., Beare, B.~K., Franchi, M., Boland, J., and Avrachenkov, K. (2025).
\newblock  \commRV{The {G}ranger–{J}ohansen representation theorem for integrated time
  series on {B}anach space.}
\newblock {\em Journal of Time Series Analysis}, 46(3):432--457.


\bibitem[Johansen, 1991]{Johansen1991}
Johansen, S. (1991).
\newblock Estimation and hypothesis testing of cointegration vectors in
  {G}aussian vector autoregressive models.
\newblock {\em Econometrica}, 59(6):1551--1580.

\bibitem[Johansen, 1992]{johansen1992representation}
Johansen, S. (1992).
\newblock A representation of vector autoregressive processes integrated of
  order 2.
\newblock {\em Econometric Theory}, 8(2):188--202.

\bibitem[Johansen, 2008]{johansen2008}
Johansen, S. (2008).
\newblock Representation of cointegrated autoregressive processes with
  application to fractional processes.
\newblock {\em Econometric Reviews}, 28(1-3):121--145.

\bibitem[Kato, 1995]{Kato1995}
Kato, T. (1995).
\newblock {\em Perturbation Theory for Linear Operators}.
\newblock Springer.

\bibitem[Kaufman, 1978]{kaufman1978representing}
Kaufman, W.~E. (1978).
\newblock Representing a closed operator as a quotient of continuous operators.
\newblock {\em Proceedings of the American Mathematical Society},
  72(3):531--534.

\bibitem[Koliha, 2014]{KOLIHA2014688}
Koliha, J. (2014).
\newblock On {K}aufman's theorem.
\newblock {\em Journal of Mathematical Analysis and Applications},
  411(2):688--692.

\bibitem[Lamond, 1989]{lamond1989generalized}
Lamond, B.~F. (1989).
\newblock A generalized inverse method for asymptotic linear programming.
\newblock {\em Mathematical Programming}, 43:71--86.

\bibitem[Lamond, 1993]{lamond1993efficient}
Lamond, B.~F. (1993).
\newblock An efficient basis update for asymptotic linear programming.
\newblock {\em Linear algebra and its applications}, 184:83--102.

\bibitem[Megginson, 1998]{megginson1998} Megginson, R.~E. (1998).
 \newblock {\em Introduction to Banach Space Theory}, \newblock{Springer New York}.

\bibitem[Nielsen et al., 2023]{NSS}
Nielsen, M.~{\O}., Seo, W.-K., and Seong, D. (2023).
\newblock Inference on the dimension of the nonstationary subspace in functional time series.
\newblock {\em Econometric Theory}, 39(3):443--480.



\bibitem[Schumacher, 1991]{Schumacher1991}
Schumacher, J.~M. (1991).
\newblock {\em System-theoretic trends in econometrics}, pages 559--577.
\newblock Springer Berlin Heidelberg, Berlin, Heidelberg.

\bibitem[Seo, 2023a]{seo_2022}
Seo, W.-K. (2023a).
\newblock Cointegration and representation of cointegrated autoregressive
  processes in {B}anach spaces.
\newblock {\em Econometric {T}heory}, 39(4):737--788.

\bibitem[Seo, 2023b]{seo_2023_fred}
Seo, W.-K. (2023b).
\newblock Fredholm inversion around a singularity: Application to
  autoregressive time series in {B}anach space.
\newblock {\em Electronic Research Archive}, 31(8):4925--4950.

\bibitem[Seo, 2024]{seo2020functional}
Seo, W.-K. (2024).
\newblock Functional principal component analysis for cointegrated functional time series.
\newblock {\em Journal of Time Series Analysis}, 45(2):320--330.

\bibitem[Seo and Beare, 2019]{seo2019cointegrated}
Seo, W.-K. and Beare, B.~K. (2019).
\newblock Cointegrated linear processes in Bayes Hilbert space.
\newblock {\em Statistics \& Probability Letters}, 147:90--95.

\bibitem[Seo and Shang, 2024]{seoshang22}
Seo, W.-K. and Shang, H.~L. (2024).
\newblock Fractionally integrated curve time series with cointegration.
\newblock {\em Electronic Journal of Statistics}, 18(2):3858--3902.

\bibitem[Seo and Shang, 2026]{seoshang25}
Seo, W.-K. and Shang, H.~L. (2026).
\newblock Testing for integer integration in functional time series.
\newblock {\em Journal of the American Statistical Association}, forthcoming (\href{https://arxiv.org/abs/2503.23960}{arXiv:2503.23960}).


\end{thebibliography}

\end{document}